\documentclass[12pt,reqno]{amsart}
\usepackage{amsmath,amssymb,amsthm,graphicx,psfrag,epsfig}
\usepackage[usenames]{color}
\usepackage{epstopdf} 

\newcommand{\BZ}{\mathbb{Z}} 
 
\newcommand{\BR}{\mathbb{R}}
\newcommand{\BP}{\mathbb{P}}

\newcommand{\CA}{\mathcal{A}}
\newcommand{\CB}{\mathcal{B}}

\newcommand{\CT}{\mathcal{T}}
\newcommand{\CX}{\mathcal{X}}

\newcommand{\co}{\colon\thinspace}

\newcommand{\la}{\langle}
\newcommand{\ra}{\rangle}

\theoremstyle{plain}
\newtheorem{theorem}{Theorem}
\newtheorem*{theoremA}{Theorem \ref{th:A}}
\newtheorem*{theoremB}{Theorem \ref{th:B}}
\newtheorem*{theoremAPP}{Theorem \ref{th:local-qg}}
\newtheorem{lemma}[theorem]{Lemma}
\newtheorem{proposition}[theorem]{Proposition}
\newtheorem{corollary}[theorem]{Corollary}

\theoremstyle{definition}
\newtheorem{definition}[theorem]{Definition}

\newtheorem{remark}[theorem]{Remark}

\numberwithin{theorem}{section}

\renewcommand{\theequation}{\thesection.\arabic{equation}}
\let\ssection=\section
\renewcommand{\section}{\setcounter{equation}{0}\ssection}

\DeclareMathOperator{\Aut}{Aut}

\DeclareMathOperator{\Out}{Out}
\DeclareMathOperator{\IA}{IA}
\DeclareMathOperator{\MCG}{MCG}
\DeclareMathOperator{\Isom}{Isom}
\DeclareMathOperator{\vol}{vol}
\DeclareMathOperator{\GL}{GL}
\DeclareMathOperator{\SL}{SL}
\DeclareMathOperator{\Sp}{Sp}

\begin{document}


\title[Current twisting and nonsingular matrices]{Current twisting and
  nonsingular matrices}

\author[M.~Clay]{Matt Clay}
\address{Dept.\ of Mathematics \\
  University of Oklahoma\\
  Norman, OK 73019}
\email{mclay@math.ou.edu}

\author[A.~Pettet]{Alexandra Pettet}
\address{Dept.\ of Mathematics \\
  University of Michigan\\
  Ann Arbor, MI 48019}
\email{apettet@umich.edu}

\thanks{\tiny The second author is partially supported by NSF grant
  DMS-0856143 and NSF RTG DMS-0602191}

\begin{abstract}
  We show that for $k \geq 3$, given any matrix in $\GL(k,\BZ)$, there
  is a hyperbolic fully irreducible automorphism of the free group of
  rank $k$ whose induced action on $\BZ^k$ is the given matrix.
\end{abstract}
\maketitle


\section{Introduction}\label{sc:intro}

Considerable progress has been made in understanding the dynamics of
elements of the outer automorphism group of a nonabelian free group of
rank $k$, $\Out F_k$, by considering the corresponding situation for the
mapping class group of a compact oriented surface of genus $g$,
$\MCG(S_g)$.  Indeed, some of the most fruitful examples of this
pedagogy include the Culler--Vogtmann Outer space $CV_k$
\cite{Culler-Vogtmann}, as well as the Bestvina--Handel train-track
representatives \cite{Bestvina-Handel}.

As a consequence of the Thurston classification of elements in
$\MCG(S_g)$, the most important elements to understand are the
\emph{pseudo-Anosov mapping classes} \cite{ar:T88}.  Such elements are
characterized as those mapping classes for which no isotopy class of a
simple closed curve in $S_g$ is periodic.  If a mapping class fixes
the isotopy class of a simple closed curve, then it restricts to a
mapping class on the subsurface obtained by cutting along the simple
closed curve.  In general, if $f \in \MCG(S_g)$, then $S_g$
decomposes into subsurfaces (which only intersect along their
boundaries) such that for some $n$, the element $f^n$ can be
represented by a homeomorphism that restricts to each subsurface as
either the identity or a pseudo-Anosov map and acts as a Dehn twist in
a neighborhood of intersection of the subsurfaces.

An element $\phi \in \Out F_k$ is \emph{fully irreducible}, also
called \emph{irreducible with irreducible powers (iwip)}, if no
conjugacy class of a proper free factor of $F_k$ is periodic.  As
above, if $\phi$ is not fully irreducible, then $F_k$ has a free
factor $F_{k'}$ such that for some $n$, the element $\phi^n$ restricts
to an element of $\Out F_{k'}$.  However, it is not the case that
$\phi^n$ preserves some free factorization of $F_k$.  The dynamics of
iterating a fully irreducible element on a conjugacy class of an
element of $F_k$ are similar to the dynamics of iterating a
pseudo-Anosov mapping class on a simple closed curve
\cite{Bestvina-Handel}.

Thurston also characterized pseudo-Anosov mapping classes as those
elements $f \in \MCG(S_g)$ whose mapping torus $S_g \times [0,1]/
(x,0) \sim (f(x),1)$ admits a hyperbolic metric \cite{ar:T88}.
However the analogous characterization for fully irreducible elements
does not hold as the mapping torus $F_k \rtimes_{\Phi} \BZ$ is not
necessarily a hyperbolic group when $\Phi \in \Aut F_k$ represents a
fully irreducible element of $\Out F_k$.  Automorphisms of $F_k$ such
that the mapping torus $F_k \rtimes_{\Phi} \BZ$ is hyperbolic are
precisely those for which no nontrivial element of $F_k$ is periodic
\cite{Bestvina-Feighn_Combination,Brinkmann,Gersten_Iso}.  Using this
correspondence, we say an element $\phi \in \Out F_k$ is
\emph{hyperbolic} if no conjugacy class of a nontrivial element of
$F_k$ is periodic.  In the literature, such elements have also been
called \emph{atoriodal}.  We remark that there are hyperbolic elements
that are not fully irreducible and fully irreducible elements that are
not hyperbolic.  However, fully irreducible elements that are not
hyperbolic have a power that is realized by a pseudo-Anosov mapping
class on a surface with a single boundary component
\cite{Bestvina-Handel}.  When $k=2$, no element of $\Out F_k$ is
hyperbolic as $\Out F_2 \cong \MCG^{\pm}(S_{1,1})$ where $S_{1,1}$ is
the torus with a single puncture.

One method to understand an element of $\MCG(S_g)$ is to examine its
action on the first homology of the surface, $H_1(S_g,\BZ) \cong
\BZ^{2g}$.  Any such element preserves the algebraic intersection
number between curves on $S_g$, giving the short exact sequence:
\[ 1 \to \mathcal{I}_g \to \MCG(S_g) \buildrel{f \mapsto
  f_*}\over\longrightarrow \Sp(2g,\BZ) \to 1. \] Similarly, the action
of an outer automorphism on $H_1(F_k,\BZ) \cong \BZ^k$ leads to the
following short exact sequence:
\[ 1 \to \IA_k \to \Out F_k \buildrel{\phi \mapsto
  \phi_*}\over\longrightarrow \GL(k,\BZ) \to 1. \]

There are various homological criteria that ensure that a given
element of the mapping class group is pseudo-Anosov
\cite{Casson-Bleiler,Koberda,Margalit-Spallone} or, in the free group
setting, that a given element of $\Out F_k$ is hyperbolic and fully irreducible
\cite{Gersten-Stallings}.  The main goal of this paper is to
generalize to the free group setting a theorem of Papadopoulos showing
that there is no homological obstruction for an element to be
pseudo-Anosov \cite{Papadopoulos}, i.e., for any $A \in \Sp(2g,\BZ)$,
there is a pseudo-Anosov mapping class $f \in \MCG(S)$ such that $f_*
= A$.

\begin{theoremA}
  Suppose $k \geq 3$.  For any $A \in \GL(k,\BZ)$, there is a
  hyperbolic fully irreducible outer automorphism $\phi \in \Out F_k$
  such that $\phi_* = A$.
\end{theoremA}

\begin{remark}\label{rm:k=2}
  For $k=2$, the function $\phi \mapsto \phi_*$ is an isomorphism and
  hence there are matrices $A \in \GL(2,\BZ)$ that are not represented
  by fully irreducible automorphisms.
\end{remark}

Papadopoulos relies on the characterization of pseudo-Anosov mapping
classes in terms of their dynamics on the Thurston boundary of
Teichm\"uller space.  The Teichm\"uller space for a surface $S_g$ is
the space of marked hyperbolic structures on $S_g$; Thurston
compactified Teichm\"uller space using the space of projectivized
measured laminations.  Pseudo-Anosovs are precisely the mapping
classes with exactly two fixed points in the compactified
Teichm\"uller space \cite{ar:T88}.  Using this characterization,
Papapodoloulos shows that if $f,h \in \MCG(S_g)$ where $f$ is
pseudo-Anosov and $f$ and $h$ satisfy an additional hypothesis, then
for large enough $m$, the mapping class $f^mh$ is pseudo-Anosov
\cite{Papadopoulos}.

Our approach for proving Theorem \ref{th:A} is similar to that of
Papadopoulos.  Namely, we show that if $\phi$ is hyperbolic and fully irreducible, and $\phi$ and $\psi \in \Out F_k$ satisfy a certain hypothesis, then for large enough $m$, the
element $\phi^m\psi$ is hyperbolic and fully irreducible (Propositions
\ref{prop:hyperbolic} and \ref{prop:fi}).  As such, one needs a space
where the dynamics of an element dictate its type, as with the action of a pseudo-Anosov 
on the Thurston boundary of Teichm\"uller space.

Since the properties of being hyperbolic and of being fully irreducible
are independent, it is perhaps of no surprise that two different
spaces are used in verifying each property for $\phi^m\psi$.  We
consider the action on the space of measured geodesic currents,
$Curr(F_k)$, as defined by Bonahon \cite{col:Bonahon91} (Section
\ref{ssc:currents}). This space is the completion of the space of
conjugacy classes for $F_k$, and thus is natural for testing
hyperbolicity.  We also consider a new complex defined by Bestvina and
Feighn for $\Out F_k$ that has the useful property of being
$\delta$--hyperbolic \cite{un:BF} (Section \ref{ssc:qm}).  Stabilizers
in $\Out F_k$ of conjugacy classes of proper free factors have bounded
orbits in this complex, and thus it provides a natural setting for
checking fully irreducibility.

Once we establish that $\phi^m\psi$ is a hyperbolic fully irreducible
element under a certain hypothesis, our problem is reduced to finding for
any $\psi \in \Out F_k$ a hyperbolic fully irreducible element $\phi
\in \IA_k$ which, together with $\psi$, satisfies the hypothesis.  To
build such elements we apply a construction from our earlier work
\cite{un:CP}; namely, we use \emph{Dehn twist automorphisms} to build
customized hyperbolic fully irreducible elements of $\Out F_k$.
Satisfying the hypothesis then requires that we understand the stable
and unstable currents in $\BP Curr(F_k)$ associated to a product of
Dehn twists.  This is our other main result, with definitions appearing in
Section \ref{sc:prelim}.

\begin{theoremB}
  Let $T_1$ and $T_2$ be very small cyclic trees that fill, with edge
  stabilizers $c_1$ and $c_2$, and with associated Dehn twist
  automorphisms $\delta_1$ and $\delta_2$. Let $N \geq 0$ be such that
  for $n \geq N$, we have that $\delta_1^n\delta_2^{-n}$ is a
  hyperbolic fully irreducible outer automorphism with stable and
  unstable currents $[\mu_+^n]$ and $[\mu_-^n]$ in $\BP Curr(F_k)$.
  Then: \[ \lim_{n \to \infty} [\mu_+^n] = [\eta_{c_1}] \mbox{ and }
  \lim_{n \to \infty} [\mu_-^n] = [\eta_{c_2}]. \]
\end{theoremB}

\bigskip \noindent {\bf Acknowledgements.}  We would like to thank Mladen
Bestvina for fielding several questions concerning this project, as well as 
Juan Souto for having suggested it as an application of our construction of hyperbolic fully irreducible outer automorphisms. We are also grateful to the referee for thoughtful and interesting suggestions concerning our results.


\section{Preliminaries}\label{sc:prelim}


\subsection{Bounded cancellation}\label{ssc:BCC}

When working with free groups, the following lemma due to Cooper is
indispensable.  For a basis $\CA$, let $|x|_\CA$ denote the word
length of $x \in F_k$ with respect to $\CA$ and $\ell_\CA(x)$ the
length of the cyclic word determined by $x$.

\begin{lemma}[\cite{Cooper}, Bounded cancellation lemma]\label{lm:BCC}
  Suppose $\CA$ and $\CB$ are bases for the free group $F_k$.
  There is a constant $C = C(\CA,\CB)$ such that if $w$ and $w'$
  are two elements of $F_k$, where:
  \[|w|_{\CA} + |w'|_{\CA} = |ww'|_{\CA}\]
  then
  \[|w|_{\CB} + |w'|_{\CB} - |ww'|_{\CB} \leq 2C.\]
\end{lemma}

\noindent We denote by $BCC(\CA,\CB)$ the bounded cancellation
constant; that is, the minimal constant $C$ satisfying the lemma for
$\CA$ and $\CB$.  In other words, if $ww'$ is a reduced word in $\CA$,
and we can write $w = \prod_{i=1}^m x_i$ and $w' =
\prod_{i=1}^{m'}x'_i$ where $x_i,x'_i \in \CB$, then for $C =
BCC(\CA,\CB)$ the subwords $x_1 \cdots x_{m-C-1}$ and $x'_{C+1}\cdots
x'_{m'}$ appear as subwords of $ww'$ when considered as a word in
$\CB$.  Applying the bounded cancellation lemma to $w^2$ where $w$ is
a cyclically reduced word with respect to $\CA$, we see that $w$ is
``almost cyclically reduced'' with respect to $\CB$, i.e., $w =
zxz^{-1}$ where $x$ is cyclically reduced with respect to $\CB$ and
$|z|_\CB \leq BCC(\CA,\CB)$.


\subsection{Culler--Vogtmann Outer space}\label{ssc:outer}

Equally indispensable to the study of $\Out F_k$ is the
\emph{Culler--Vogtmann Outer space} $CV_k$ \cite{Culler-Vogtmann}.
This is the projectivized space of minimal discrete free actions of
$F_k$ on $\BR$--trees and is analogous to the Teichm\"uller space for
a surface.  There is a compactification $\overline{CV}_k$
\cite{Culler-Morgan} that is precisely the projectivized space of
minimal very small actions of $F_k$ on $\BR$--trees
\cite{Bestvina-Feighn,Cohen-Lustig}.  Recall that an action on an
$\BR$--tree is minimal if there is no invariant subtree; it is very
small if the stabilizer of an arc is either trivial or a maximal
cyclic subgroup, and if the stabilizer of any tripod is trivial.  We
consider the unprojectivized versions $cv_k$ and $\overline{cv}_k$ as
well.

The group $\Out F_k$ acts on either of the above spaces on the right
by pre-composing the action homomorphism.  Fully irreducible elements
act on $\overline{CV}_k$ with North-South dynamics.

\begin{theorem}[\cite{Levitt-Lustig}, Theorem~1.1]\label{th:NS-cv}
  Every fully irreducible element $\phi \in \Out F_k$ acts on
  $\overline{CV}_k$ with exactly two fixed points $[T_+]$ and $[T_-]$.
  Further, for any $[T] \in \overline{CV}_k$ such that $[T] \neq
  [T_-]$:
  \[ \lim_{m \to \infty} [T\phi^m] = [T_+]. \]
\end{theorem}

The trees $[T_+]$ and $[T_-]$ are called the \emph{stable} and
\emph{unstable} trees of $\phi$ respectively.  The stable and unstable
trees of $\phi^{-1}$ are $[T_-]$ and $[T_+]$, respectively.


\subsection{Dehn twists}\label{ssc:dt}

As mentioned in the introduction, we build customized hyperbolic fully
irreducible elements of $\Out F_k$ using Dehn twist automorphisms.
These are defined analogously to a Dehn twist homeomorphism of a
surface.  Specifically, given a splitting $F_k = A*_{\la c \ra} B$, we
define an automorphism by:
\begin{align*}
  \forall a \in A \qquad & \delta(a) = a \\
  \forall b \in B \qquad & \delta(b) = cbc^{-1}.
\end{align*}
The automorphism $\delta$ acts trivially on homology and therefore
belongs to the subgroup $\IA_k$. A Dehn twist automorphism arising
from amalgamations over $\BZ$ is analogous to a Dehn
twist around a separating simple closed curve on a surface.

We similarly obtain an automorphism $\delta$ from an HNN-extension of
the form
$$F_k = A *_\BZ = \la A, t  \ | \ t^{-1}a_0t = a_1 \ra$$
for $a_0, a_1 \in A$ by:
\begin{align*}
  \forall a \in A \qquad & \delta(a) = a \\
  & \delta(t) = a_0t.
\end{align*}
An automorphism arising from an HNN-extension should be compared to a
Dehn twist around a nonseparating curve on a surface.

From Bass-Serre theory, a splitting of $F_k$ over $\BZ$ defines an
action of $F_k$ on a tree $T$, the {\it Bass-Serre tree} of the
splitting (see \cite{Bass} or \cite{Serre}). We will refer to such
$F_k$-trees as {\it cyclic}.  In a
certain sense, cyclic trees for $F_k$ correspond to simple closed
curves on a surface; as in the mapping class group, the Dehn twist
automorphisms associated to cyclic trees generate an index two
subgroup of $\Aut F_k$ (the subgroup which induces an action of
$\SL_k(\BZ)$ on homology). Note that if $\delta$ is the Dehn twist
automorphism associated to the cyclic tree $T$, then $\delta$
preserves the action of $F_k$ on $T$, i.e.,~there is an isometry
$h_\delta\co T \to T$ such that $\forall g \in F_k$ and $\forall x \in
T$ we have $h_\delta(gx) = \delta(g)h_\delta(x)$.  In particular,
$\ell_T(\delta(g)) = \ell_T(g)$ for all $g \in F_k$.

We are primarily interested in the {\it outer} automorphism group of
$F_k$, and so in the sequel a Dehn twist will refer to an element of
$\Out F_k$ which is induced by a Dehn twist automorphism in $\Aut
F_k$.

The role of intersection number of simple closed curves is played by
\emph{free volume}.

\begin{definition}[Free volume]\label{defi:volume}
  Suppose $X$ is a finitely generated free group that acts on a
  simplicial tree $T$ such that the stabilizer of an edge is either
  trivial or cyclic. The \emph{free volume $\vol_T(X)$} of $X$ with
  respect to $T$ is the number of edges in the graph of groups
  decomposition $T^X/X$ with trivial stabilizer.  Here $T^X$ denotes
  the smallest $X$--invariant subtree.
\end{definition}
\noindent In the case that $X = \la x \ra$, the free volume $\vol_T(X)$ is just the 
\emph{translation length} $\ell_T(x)$ of $x$ in $T$.  

We say two cyclic trees
\emph{fill} if:
\begin{equation*}\label{eq:F1}
  \vol_{T_1}(X) + \vol_{T_2}(X) > 0
\end{equation*}
for every proper free factor or cyclic subgroup $X \subset F_k$.  With
these notions we have shown the following analog to a classical
theorem of Thurston:
\begin{theorem}[\cite{un:CP}, Theorem~5.3]\label{th:twist}
  Let $\delta_1$ and $\delta_2$ be the Dehn twist automorphisms of
  $F_k$ for two filling cyclic trees of $F_k$.  Then there
  exists $N = N(\delta_1,\delta_2)$ such that for all $m,n \geq N$:
  \begin{enumerate}
  \item $\la \delta_1^m, \delta_2^n \ra$ is isomorphic to the free
    group on two generators; and
  \item if $\phi \in \la \delta_1^m, \delta_2^n \ra$ is not conjugate
    to a power of either $\delta_1^m$ or $\delta_2^n$, then $\phi$ is
    a hyperbolic fully irreducible element of $\Out F_k$.
  \end{enumerate}
\end{theorem}

Key to our analysis in \cite{un:CP} and Section \ref{sc:stable} of the
present paper is the following theorem, which measures how the free
volume changes upon twisting.

\begin{theorem}[\cite{un:CP}, Theorem~4.6]\label{th:linear}
  Let $\delta_2$ be a Dehn twist automorphism corresponding to a 
  very small cyclic tree $T_2$ with cyclic edge generator $c_2$, and let
  $T_1$ be any other very small cyclic tree.  Then there is a constant
  $C = C(T_1,T_2)$ such that for any $x \in F_k$ and $n \geq 0$ the
  following hold.
  \begin{align}
    \ell_{T_1}(\delta_2^{\pm n}(x)) & \geq
    \ell_{T_2}(x)\bigl[n\ell_{T_1}(c_2) - C \bigr] - \ell_{T_1}(x)
    \label{al:twist-lb} \\
    \ell_{T_1}(\delta_2^{\pm n}(x)) & \leq
    \ell_{T_2}(x)\bigl[n\ell_{T_1}(c_2) + C \bigr] + \ell_{T_1}(x)
    \label{al:twist-ub}
  \end{align}
\end{theorem}

\noindent These bounds are shown in \cite{un:CP} to hold not only for cyclic subgroups, but for any finitely generated malnormal subgroup of $F_k$; in particular any proper free factor of $F_k$.

We will also need the following notions from \cite{un:CP} for Section
\ref{sc:stable}.

Suppose that $T$ is a very small cyclic tree for an amalgamated free
product $F_k = A*_{\la c \ra} B$.  After possibly interchanging $A
\leftrightarrow B$, there is a basis $\CT = \CA \cup \CB$ for $F_k$
such that $c \in \CA$, and such that $\CA$ is a basis for $A$ and $\CB
\cup \{ c \}$ is a basis for $B$.  Such a basis is called a
\emph{basis relative to $T$}.  If $x \in F_k$ and $\ell_T(x) = 2m >
0$, then $x$ is conjugate to a cyclically reduced word of the form:
\begin{equation*}
  x_1c^{i_1}y_1c^{j_1} \cdots x_mc^{i_m}y_mc^{j_m}
\end{equation*}
where for $s =1,\ldots,m$, each $y_s$ is a word in $\CB$, each
$x_s$ a word in $\CA$, such that both $zx_s$ and $x_sz$ are reduced
for $z = c,c^{-1}$.

Now suppose that $T$ is a very small cyclic tree for an HNN-extension
$F_k = A*_{\la tc't^{-1} = c \ra}$.  After possibly interchanging $A
\leftrightarrow tAt^{-1}$, there is a basis $\CA \cup \{ t_0\}$ for
$F_k$ such that $t = t_0a$ for some $a \in A$, $c \in \CA$ and $\CA
\cup \{ t_0^{-1}ct_0 \}$ is a basis for $A$.  If $x \in F_k$ and
$\ell_T(x) = m > 0$, then $x$ is conjugate to a cyclically reduced
word of the form:
\begin{equation*}
  x_1(c^{i_1}t_0)^{\epsilon_1}x_2(c^{i_2}t_0)^{\epsilon_2} \cdots 
  x_m(c^{i_m}t_0)^{\epsilon_m}  
\end{equation*}
where for $s = 1, \ldots,m$, $x_s$ is a word in $\CA \cup
\{t_0^{-1}ct_0 \}$, $\epsilon_s \in \{ \pm 1\}$; and if $\epsilon_s =
1$, then $x_sz$ is a reduced word for $z = c,c^{-1}$; and if
$\epsilon_r = -1$ then $zx_{s+1}$ is a reduced word for $z =
c,c^{-1}$.

In either of two above cases, we say that the specific word is
\emph{$T$--reduced}.


\subsection{Currents}\label{ssc:currents}

Measured geodesic currents for hyperbolic groups were first defined by
Bonahon \cite{col:Bonahon91}.  Recently, (measured geodesic) currents
for free groups have seen much activity through the work of Kapovich
and Lustig
\cite{col:Kapovich06,KL_Complex,KL_Intersection,KL_Boundary}.  We
briefly introduce the parts of the theory needed for the sequel; see
\cite{col:Kapovich06} for further details.

The group $F_k$ is hyperbolic and hence has a boundary $\partial F_k$.
We denote:
\[ \partial^2 F_k = \{ (x_1,x_2) \in \partial F_k \times \partial F_k
\; | \; x_1 \neq x_2 \} \] This is naturally identified with the space
of oriented geodesics in a Cayley tree for $F_k$.  There is
fixed-point free involution ``flip'' map $\sigma\co \partial^2 F_k
\to \partial^2 F_k$ defined by $\sigma(x_1,x_2) = (x_2,x_1)$ which
corresponds to reversing the orientation on the geodesic.

A {\it (measured geodesic) current} on $F_k$ is an $F_k$--invariant
positive Radon measure on $\partial^2 F_k/\sigma$.  The set
$Curr(F_k)$ is the set of all currents on $F_k$, topologized with the
weak topology.  There is an action of $\BR_{> 0}$ on $Curr(F_k) -\{ 0
\}$, and the quotient $\BP Curr(F_k)$ is a compact space.  There is a
continuous left action of $\Out F_k$ on $Curr(F_k)$ and $\BP
Curr(F_k)$ defined by $\phi\nu(S) = \nu(\phi^{-1}(S))$, where $\phi
\in \Out F_k$, $\nu \in Curr(F_k)$, and where $S$ is a measurable set
of $\partial^2 F_k/\sigma$.  There is a slight abuse of notation here
as strictly speaking $\phi^{-1}(S)$ is not well-defined.  But for any
two $\Phi_0,\Phi_1 \in \Aut F_k$ representing $\phi \in \Out F_k$,
there is an $x \in F_k$ such that $x\Phi_0^{-1}(S) = \Phi_1^{-1}(S)$
and hence $\nu(x\Phi_0^{-1}(S)) = \nu(\Phi_1^{-1}(S))$ since $\nu$ is
$F_k$--invariant.

Given a basis $\CA$ of $F_k$, we have an identification between
$\partial^2 F_k/\sigma$ and unoriented geodesics in $T_\CA$, the
Cayley tree for $\CA$. For a nontrivial $g \in F_k$ (thought of as a
vertex in $T_\CA$) and $\nu \in Curr(F_k)$, we define the
\emph{two-sided cylinder}:
\begin{align*} 
  Cyl_\CA(g) & = \{\mbox{unoriented geodesics in $T_\CA$ containing} \\
  & \qquad \qquad \mbox{the vertices $1$ and $g$} \}
  \subset \partial^2 F_k/\sigma
  \mbox{ and denote} \\
  \la g, \nu \ra_\CA & = \nu(Cyl_\CA(g)).
\end{align*}
As the sets $\bigcup_{h,g \in F_k} hCyl_\CA(g)$ form a basis for the
topology of $\partial^2 F_k/\sigma$, and as $\nu(hCyl_\CA(g)) =
\nu(Cyl_\CA(g))$, a current $\nu \in Curr(F_k)$ is determined by its
values $\la g, \nu \ra_\CA$.
 
Using these notions there is a useful normalization of a current $\nu$
relative to the basis $\CA$.  Put:
\[ \omega_\CA(\nu) = \sum_{x \in \CA} \la x, \nu \ra_\CA. \] The
following lemma provides a useful way to show convergence in $\BP
Curr(F_k)$:
\begin{lemma}[\cite{col:Kapovich06}, Lemmas 2.11 and
  3.5]\label{lm:convergence}
  Let $\CA$ be a basis for $F_k$.  Then $\lim_{m \to \infty}[\nu_m] =
  [\nu]$ if and only if for every nontrivial $g \in F_k$
  \[ \lim_{m \to \infty} \frac{\la g, \nu_m\ra_\CA}{\omega_\CA(\nu_m)}
  = \frac{\la g,\nu\ra_\CA}{\omega_\CA(\nu)}. \]
\end{lemma}

Particularly useful are the \emph{counting currents}, defined as
follows.  Given a nontrivial $h \in F_k$ that is not a proper power,
define the current $\eta_h$ by:
\[ \la g, \eta_h \ra_\CA = \la g^{\pm 1}, h \ra_\CA. \] Here $\la
g^{\pm 1}, h \ra_\CA$ is the number of \emph{occurrences} of $g$ or
$g^{-1}$ in the cyclic word determined by $h$; specifically, this is
the number of times either of the reduced words $g$ or $g^{-1}$ appear
as a subword of the cyclic word determined by $h$.  When $h = f^m$
where $m \geq 1$ and $f$ is not a proper power, define $\eta_h =
m\eta_f$.  The current $\eta_h$ only depends on the conjugacy class of
$h$, and for $\phi \in \Out F_k$ we have $\phi\eta_h =
\eta_{\phi(h)}$.  Notice that for any nontrivial $h \in F_k$ we have
$\omega_\CA(\eta_h) = \ell_\CA(h)$.  Although we will not explicitly
use it, we remark that the set $\{[\eta_h]\}_{h \in F_k - \{1\}}$ is
dense in $\BP Curr(F_k)$.

Similarly we define $o(g^{\pm 1},h)_\CA$ as the number occurrences of
$g$ or $g^{-1}$ in the word $h$; specifically, this is the number of
times the reduced words $g$ or $g^{-1}$ appear as a subword of the
word $h$.  A direct application of the Bounded Cancellation Lemma
\ref{lm:BCC} gives the following.

\begin{lemma}\label{lm:occurance-cancellation}
  Let $\CA$ and $\CB$ be bases for $F_k$ and fix $a \in \CA$.  Then
  there exists a constant $C \geq 0$ such that if $w$, $w'$, and $ww'$
  are all reduced words in $\CB$ and $ww'$ is cyclically reduced in
  $\CB$, then:
  \begin{equation*}
    o(a^{\pm 1}, w )_\CA + o(a^{\pm 1}, w' )_\CA \leq
    \la a^{\pm 1}, ww' \ra_\CA + C.
  \end{equation*}
\end{lemma}

\begin{proof}
  Let $B = BCC(\CB,\CA)$ so that:
  \[ o( a^{\pm 1}, w )_\CA + o( a^{\pm 1}, w' )_\CA - 2B \leq o(a^{\pm
    1},ww')_\CA. \] Since $ww'$ is cyclically reduced with respect to
  $\CB$, as a word in $\CA$ we have $ww' = zxz^{-1}$ where $|z|_\CA
  \leq B$ and $x$ is cyclically reduced in $\CA$.  Thus:
  \[ o(a^{\pm 1},ww')_\CA \leq \la a^{\pm 1}, ww' \ra_\CA + 2B. \]
  Therefore, for $C = 4B$, the lemma holds.
\end{proof}

As in the Outer space setting, a hyperbolic fully irreducible element
acts with North-South dynamics on $\BP Curr(F_k)$.  Here is a weak
version of this statement that is sufficient for our needs.

\begin{theorem}[\cite{Martin}, cf.~\cite{un:BF},~Proposition
  4.11]\label{th:NS-curr}
  Every hyperbolic fully irreducible element $\phi \in \Out F_k$ acts
  on $\BP Curr(F_k)$ with exactly two fixed points, $[\mu_+]$ and
  $[\mu_-]$.  Further, for any nontrivial $h \in F_k$:
  \[ \lim_{m \to \infty} [\phi^m\eta_h] = [\mu_+]. \]
\end{theorem}

The currents $[\mu_+]$ and $[\mu_-]$ are called the \emph{stable} and
\emph{unstable} currents of $\phi$, respectively.  The stable and
unstable currents of $\phi^{-1}$ are $[\mu_-]$ and $[\mu_+]$,
respectively.

The existence of a continuous $\Out F_k$--invariant {\it intersection
  form} is established by the following.

\begin{theorem}[\cite{KL_Complex}, Theorem A]\label{th:intersection}
  There is a unique continuous map:
  \[ \la \; \, , \; \ra\co \overline{cv}_k \times Curr(F_k) \to \BR_{\geq
    0} \]
  such that:
  \begin{enumerate}
  \item for any $h \in F_n$ we have $\la T, \eta_h \ra = \ell_T(h)$; and which is 
  \item $\Out F_k$--invariant: $\la T\psi,\mu \ra = \la T,\psi
    \mu \ra$;
  \item homogeneous with respect the the first coordinate: $\la
    \lambda T, \mu \ra = \lambda \la T ,\mu \ra$ for $\lambda > 0$;
    and
  \item linear with respect to the second coordinate: $\la
    T,\lambda_1\mu_1 + \lambda_2\mu_2 \ra = \lambda_1\la T,\mu_1 \ra +
    \lambda_2\la T,\mu_2 \ra$ for $\lambda_1,\lambda_2 \geq 0$.
  \end{enumerate}
\end{theorem}

The actions of $\Out F_k$ on $\overline{cv}_k$ and $Curr(F_k)$ satisfy
a type of ``unique-ergodicity'' with respect to this intersection form.

\begin{theorem}[\cite{KL_Intersection}, Theorem 1.3]\label{th:ergodic}
  Let $\phi \in \Out F_k$ be a hyperbolic fully irreducible element
  with stable and unstable trees $[T_+], [T_-] \in \overline{CV}_k$
  and stable and unstable currents $[\mu_+], [\mu_-] \in \BP
  Curr(F_k)$.  The following statements hold.
  \begin{enumerate}

  \item If $\mu \in Curr(F_k) - \{ 0 \}$, then $\la T_\pm,\mu \ra = 0$
    if and only if $[\mu] = [\mu_\mp]$.

  \item If $T \in \overline{cv}_k$, then $\la T, \mu_\pm \ra = 0$ if
    and only if $[T] = [T_\mp]$.

  \end{enumerate}
\end{theorem}

The difference in signs $\pm$ and $\mp$ between the above and
its version in \cite{KL_Intersection} is due to our use of the right
action of $\Out F_k$ on $\overline{cv}_k$.


\subsection{\texorpdfstring{Bestvina--Feighn hyperbolic
    $\Out(F_k)$--complex}{Bestvina--Feighn hyperbolic
    Out(Fk)--complex}}\label{ssc:qm}

The final space we consider is given by the following theorem.

\begin{theorem}[\cite{un:BF}, Main Theorem]\label{th:complex}
  For any finite collection $\phi_1,\ldots,\phi_n$ of fully
  irreducible elements of $\Out F_k$ there is a connected
  $\delta$--hyperbolic graph $\CX$ equipped with an (isometric) action
  of $\Out F_k$ such that:
  \begin{enumerate}
  \item the stabilizer in $\Out F_k$ of a simplicial tree in
    $\overline{CV}_k$ has bounded orbits;
  \item the stabilizer in $\Out F_k$ of a proper free factor $F
    \subset F_k$ has bounded orbits; and
  \item $\phi_1,\ldots,\phi_n$ have nonzero translation lengths.
  \end{enumerate}
\end{theorem}

The $\delta$-hyperbolicity of such a complex $\CX$ makes it comparable to the 
curve complex for the mapping class group, although its use is significantly restricted by its dependence on a finite set of fully irreducible elements.  For our purposes the actual definition of $\CX$ is not necessary; we need only that non-fully irreducible elements of $\Out F_k$ act on $\CX$ with bounded orbits,
and that the action of the elements $\phi_1,\ldots,\phi_n$ on $\CX$ have
nonzero translation length and satisfy a property known as WPD (weak
proper discontinuity).  We refer the reader to \cite{un:BF,ar:BF02} for further details.


\section{Producing hyperbolic automorphisms}\label{sc:hyperbolic}

In this section we show how to produce a hyperbolic outer automorphism
with a specified action on $H_1(F_k,\BZ)$.  
This involves examining the dynamics of elements on $Curr(F_k)$. Using
the ``unique-ergodicity'' and continuity of the intersection form $\la
\; ,\; \ra$ we can mimic an argument due to Fathi
\cite[Theorem~2.3]{ar:Fathi87} giving a construction of pseudo-Anosov
homeomorphisms.

\begin{proposition}\label{prop:hyperbolic}
  Let $\phi \in \Out F_k$ be a hyperbolic fully irreducible outer
  automorphism with stable and unstable currents $[\mu_+]$ and
  $[\mu_-]$ in $\BP Curr(F_k)$.  Suppose $\psi \in \Out F_k$ is such
  that $[\psi\mu_+] \neq [\mu_-]$.  Then there is an $M \geq 0$ such
  that for $m \geq M$ the element $\phi^m\psi$ is hyperbolic.
\end{proposition}

\begin{proof}
  Let $\lambda_+$ and $\lambda_-$ be the expansion factors for $\phi$
  and $\phi^{-1}$ respectively, and let $\lambda = \min \{ \lambda_+,
  \lambda_-\} > 1$.  Also let $T_+$ and $T_-$ be representatives of
  the stable and unstable trees for $\phi$ in $\overline{cv}_k$.  Thus
  $T_+\phi = \lambda_+T_+$ and $T_-\phi^{-1} = \lambda_-T_-$.

  Hence for each $m \geq 0$ and any $\mu \in Curr(F_k)$ we have:
  \begin{align*}
    \la T_+,\phi^m \psi \mu \ra & = \la T_+\phi^m, \psi\mu
    \ra \geq \lambda^m \la T_+, \psi\mu \ra, \mbox{ and} \\
    \la T_-\psi, \psi^{-1}\phi^{-m} \mu \ra & = \la T_-,\phi^{-m} \mu
    \ra = \la T_-\phi^{-m}, \mu \ra \geq \lambda^m \la T_-,\mu \ra.
  \end{align*}
  Now define $\alpha(\mu) = \max\{\la T_+, \mu \ra, \la T_-\psi , \mu
  \ra \}$.  Then:
  \begin{align*}
    \alpha(\phi^m\psi\mu) & \geq \la T_+, \phi^m\psi \mu \ra \geq
    \lambda^m \la T_+, \psi\mu \ra, \mbox{ and} \\
    \alpha(\psi^{-1}\phi^{-m}\mu) & \geq \la T_-\psi,
    \psi^{-1}\phi^{-m} \mu \ra \geq \lambda^m \la T_-, \mu \ra.
  \end{align*}
  Hence $\max\{ \alpha(\phi^m\psi\mu),\alpha(\psi^{-1}\phi^{-m}\mu) \}
  \geq \lambda^m\beta(\mu)$, where $\beta(\mu) = \max\{ \la T_+,
  \psi\mu \ra, \la T_-,\mu \ra \}$.  Now $\beta(\mu) = 0$ if and only
  if both $\la T_+,\psi\mu \ra$ and $\la T_-,\mu \ra$ are equal to 0.
  Applying the ``unique-ergodicity'' (Theorem \ref{th:ergodic}), we
  have that if $\mu \neq 0$ then $\la T_-, \mu \ra = 0$ if and only if
  $[ \mu ] = [\mu_+ ]$, and $\la T_+,\psi\mu_+ \ra = 0$ if and only if
  $[\psi\mu_+] = [\mu_-]$.  By assumption $[\psi\mu_+] \neq [\mu_-]$, and
  hence $\beta(\mu)$ is strictly positive.  Therefore
  $\alpha(\mu)/\beta(\mu)$ defines a continuous function on $\BP
  Curr(F_k)$.  Since $\BP Curr(F_k)$ is compact, there is a constant
  $K$ such that $\alpha(\mu)/\beta(\mu) < K$ for all $\mu \in
  Curr(F_k) - \{ 0 \}$, i.e., $K\beta(\mu) > \alpha(\mu)$.  For $m$
  such that $\lambda^m \geq K$, we obtain:
  \[ \forall \mu \in Curr(F_k) - \{ 0 \}, \quad \max \{
  \alpha(\phi^m\psi\mu), \alpha(\psi^{-1}\phi^{-m}\mu) \} >
  \alpha(\mu). \] It is now easy to see that $\phi^m\psi$ acts on
  $Curr(F_k) - \{ 0 \}$ without a periodic orbit.

  Notice that if $\theta \in \Out F_k$ has a periodic conjugacy class,
  say $\theta^\ell$ fixes the conjugacy class of $c$, then
  $\theta^\ell\eta_c = \eta_{\theta^\ell(c)} = \eta_c$, and hence
  $\theta$ acts on $Curr(F_k) -\{ 0 \}$ with a periodic orbit.  Thus
  as $\phi^m\psi$ acts on $Curr(F_k) - \{ 0 \}$ without a periodic
  orbit it does not have a periodic conjugacy class, i.e.,
  $\phi^m\psi$ is hyperbolic.
\end{proof}


\section{Producing fully irreducible automorphisms}\label{sc:fi}

In this section we show how to produce a fully irreducible element of
$\Out F_k$ with a specified action on $H_1(F_k,\BZ)$.  This involves
examining the dynamics of elements on the $\delta$--hyperbolic
Bestvina--Feighn complex $\CX$ from Theorem \ref{th:complex}.  We
begin with a theorem about the isometries of $\delta$--hyperbolic
spaces.  Even though the space we will ultimate consider has a right
action, we will consider the more customary setting where the space
has a left action; it is clear how to convert a right action into a
left action.

We recall some basics about $\delta$--hyperbolic spaces needed for
this section.  Some references for this material are
\cite{col:HG91,bk:BH99,col:Gr87}.

A geodesic metric space $X$ is called \emph{$\delta$--hyperbolic} if
for any geodesic triangle in $X$, the $\delta$--neighborhood of the
union of any two of the sides contains the third.  There are various
other equivalent notions.  There is an \emph{inner product} defined
for points $x,y \in X$ by:
\[ (x . y)_w = \frac{1}{2}(d(x,w) + d(w,y) - d(x,y)) \] for a given
basepoint $w \in X$.  Associated to a $\delta$--hyperbolic space is a
\emph{boundary} $\partial X$ which compactifies $X$ as $X
\cup \partial X$ when $X$ is locally compact.  One definition of
$\partial X$ is as equivalence classes of sequences $\{ x_i \}$ with
$\lim_{i,j \to \infty} (x_i.x_j) = \infty$ (the inner product is defined with
respect to some basepoint), the equivalence relation is defined by $\{
x_i \} \sim \{ y_i \}$ if $\lim_{i \to \infty} (x_i.y_i) = \infty$.
If $f$ is an isometry of $X$ with nonzero translation length
(i.e.,~$\lim_{n \to \infty} \frac{1}{n}d(x,f^n(x)) > 0$ for all $x \in
X$), then the action of $f$ extends to a continuous action on
$\partial X$ with exactly two fixed points.  One fixed point is
represented by the sequence $\{ f^n(x) \}$ for any $x \in X$; the
other is represented by $\{ f^{-n}(y) \}$ for any $y \in X$.  These
points are called the \textit{attracting and repelling fixed points}
of $f$ respectively.

\begin{theorem}\label{th:hyperbolic-action}
  Suppose $X$ is a $\delta$--hyperbolic space and $f \in \Isom(X)$
  acts on $X$ with nonzero translation length, with attracting and
  respectively repelling fixed points $A_+$ and $A_-$ in $\partial X$.
  If $g \in \Isom(X)$ acts on $X$ such that $gA_+ \neq A_-$, then
  there is an $M \geq 0$ such that for $m \geq M$ the element $f^mg$
  acts on $X$ with nonzero translation length.
\end{theorem}

Before proving this theorem we need a lemma that allows us to locally
build uniform quasi-geodesics.  Recall that a
\textit{$(\lambda,\epsilon)$--quasi-geodesic} is a function $\alpha\co
[a,b] \to X$ such that for all $t,t' \in [a,b]$ we have:
\[ \frac{1}{\lambda}|t - t'| - \epsilon \leq d(\alpha(t),\alpha(t'))
\leq \lambda |t - t'| + \epsilon.  \] We allow for the possibility that
the domain of $\alpha$ is $\BR$ or $\BR_{\geq 0}$.  A function
$\alpha\co[a,b] \to X$ is an \textit{$L$--local
  $(\lambda,\epsilon)$--quasi-geodesic} if for all $a \leq a' \leq b'
\leq b$ where $b' - a' \leq L$, the function $\alpha \big|_{[a',b']}$
is a $(\lambda,\epsilon)$--quasi-geodesic. First we recall a standard
fact about $\delta$--hyperbolic spaces.

\begin{lemma}[\cite{bk:BH99}, Chapter III.H Lemma
  1.15]\label{lm:BH-lemma}
  Let $X$ be a $\delta$--hyperbolic space, and let $c_1\co [0,T_1] \to X$
  and $c_2\co[0,T_2] \to X$ be geodesics such that $c_1(0) = c_2(0)$.
  Let $T = \max\{T_1,T_2 \}$ and extend the shorter geodesic to
  $[0,T]$ by the constant map.  If $K = d(c_1(T),c_2(T))$, then
  $d(c_1(t),c_2(t)) \leq 2(K + 2\delta)$ for all $t \in [0,T]$.
\end{lemma}

The next lemma shows us that the sequence of points $(f^mg)^n(x)$
defines a local quasi-geodesic with uniform constants.

\begin{lemma}\label{lm:quasi-geodesic}
  Let $X$, $f$ and $g$ be as in Theorem \ref{th:hyperbolic-action}.
  Fix $x \in X$ and for $m \geq 0$, let $\alpha_m$ be a geodesic
  connecting $x$ to $f^mg(x)$.  Then there is an $\epsilon \geq 0$
  such that for $m \geq 0$ the concatenation of the geodesics
  $\alpha_m \cdot f^mg(\alpha_m)$ is a $(1,\epsilon)$--quasi-geodesic.
\end{lemma}

\begin{proof}
  Let $\beta_m = \alpha_m \cup f^mg(\alpha_m)$, $d_m = d(x,f^mg(x))$
  and consider the points $g(x)$, $f^{-m}(x)$ and $gf^mg(x)$.  Notice
  $f^{-m}(\beta_m)$ is a path from $f^{-m}(x)$ to $gf^mg(x)$ passing
  through $g(x)$; see Figure \ref{fig:hyperbolic}. As $gA_+ \neq A_-$,
  the inner product $(f^{-m}(x).gf^mg(x))_{g(x)}$ stays bounded as $m
  \to \infty$.  Hence there is a constant $C \geq 0$ that does not
  depend on $m$ such that:
  \begin{align*}
    d(x,f^mgf^mg(x)) & = d(f^{-m}(x),gf^mg(x)) \\
    & \geq d(f^{-m}(x),g(x)) + d(g(x),gf^mg(x)) -
    2C \\
    & = 2d_m - 2C.
  \end{align*}
  \begin{figure}[t]
    \centering
    \iftrue
      \psfrag{+}{$gA_+$}
      \psfrag{-}{$A_-$}
      \psfrag{a}{$gf^mg(x)$}
      \psfrag{b}{$f^{-m}(x)$}
      \psfrag{c}{$g(x)$}
      \psfrag{d}{$f^m$}
      \psfrag{z}{$z$}
      \psfrag{e}{$f^mgf^mg(x)$}
      \psfrag{x}{$x'$}
      \psfrag{f}{$f^mg(x)$}
      \psfrag{g}{$x$}
      \psfrag{A}{$\alpha_m$}
      \psfrag{D}{$\delta$}
      \psfrag{C}{$C+2\delta$}
      \includegraphics{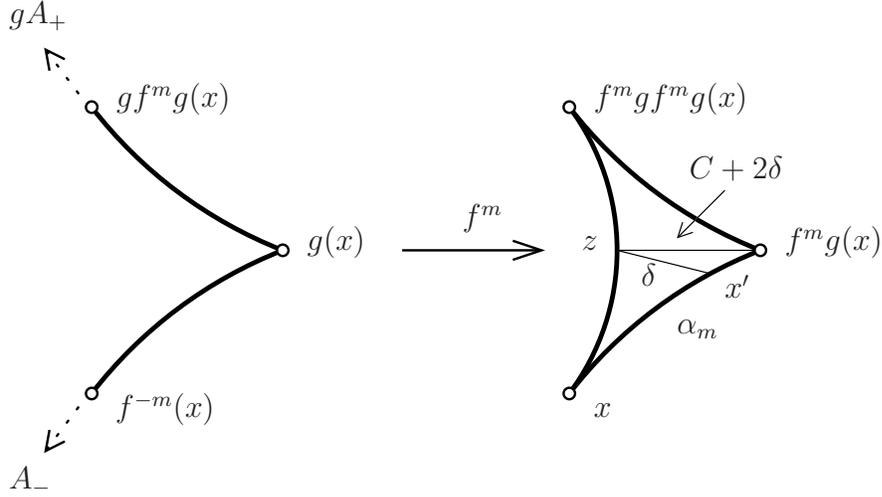}
    \fi
    \caption{The geodesics in Lemma \ref{lm:quasi-geodesic}.}
    \label{fig:hyperbolic}
  \end{figure}

  Fix a geodesic $c$ from $x$ to $f^mgf^mg(x)$, and let $z$ be the
  midpoint on $c$.  As $X$ is $\delta$--hyperbolic, there is an $x' \in
  \beta_m$ such that $d(z,x') \leq \delta$.  Without loss of
  generality we can assume that $x' \in \alpha_m$.  Thus:
  \[ d(x',x) \geq d(x,z) - \delta \geq d_m - C -
  \delta \]
  and therefore:
  \[ d(x',f^mg(x)) = d_m - d(x,x') \leq C + \delta, \] from which we
  conclude $d(z,f^mg(x)) \leq C + 2\delta$.  Let $d'_m =
  d(x,f^mgf^mg(x))$ and define $c_z\co [0,d_m] \to X$ by $c_z(t) =
  c(t)$ if $0 \leq t \leq \frac{1}{2}d'_m$ and $c_z(t) = z$ otherwise.
  Then by Lemma \ref{lm:BH-lemma} we have for $0 \leq t \leq d_m$ that
  $d(\beta_m(t),c_z(t)) \leq 2(C + 4\delta)$.  Similarly define
  $c_z'\co [d_m,2d_m] \to X$ by $c_z'(t) = z$ if $d_m \leq t \leq 2d_m
  - \frac{1}{2}d'_m$ and $c_z'(t) = c(t + d'_m - 2d_m)$ otherwise.
  Then another application of Lemma \ref{lm:BH-lemma} shows that for
  $d_m \leq t \leq 2d_m$ we have $d(\beta_m(t),c_z'(t)) \leq 2(C +
  4\delta)$.  Notice that if $0 \leq t \leq d_m \leq t' \leq 2d_m$
  then:
  \[ (t' - t) - 2C \leq d(c_z(t),c'_z(t')) \leq (t' - t) \] as $2d_m -
  d'_m \leq 2C$. Therefore if $0 \leq t \leq d_m \leq t' \leq 2d_m$
  then:
  \[ (t' - t) - (6C + 16\delta) \leq d(\beta_m(t),\beta_m(t')) \leq
  (t' - t). \] The other cases ($0 \leq t \leq t' \leq d_m$ or $d_m
  \leq t \leq t' \leq 2d_m$) are clear since $\alpha_m$ is a geodesic.
\end{proof}

Now to complete the proof of Theorem \ref{th:hyperbolic-action} we
need the following theorem.  

\begin{theorem}[\cite{bk:BH99}, Chapter III.H Theorems 1.7 \&
  1.13]\label{th:local-qg}
  Let $X$ be a $\delta$--hyperbolic space, and let $\gamma\co [a,b] \to
  X$ be an $L$--local $(\lambda,\epsilon)$--quasi-geodesic.  Then
  there is an $R = R(\delta,\lambda,\epsilon)$ such that if $L > R$,
  then for some $\lambda' \geq 1$ and $\epsilon' \geq 0$, the path
  $\gamma$ is a $(\lambda',\epsilon')$--quasi-geodesic.
\end{theorem}

We can now give a proof of Theorem \ref{th:hyperbolic-action}.

\begin{proof}[Proof of Theorem \ref{th:hyperbolic-action}]
  Fix $x \in X$, and let $\epsilon$ be given from Lemma
  \ref{lm:quasi-geodesic} and let $R = R(\delta,1,\epsilon)$ be the
  constant from Theorem \ref{th:local-qg}.  As $f$ has nonzero translation length, for $m \geq 0$ we can let $L_m =
  d(x,f^mg(x)) \geq d(g(x),f^mg(x)) - d(x,g(x)) \geq mt - d(x,g(x))$
  for some $t > 0$.  Let $M$ be
  such that $L_M > R$.  As in Lemma \ref{lm:quasi-geodesic}, let
  $\alpha_m$ be a geodesic connecting $x$ to $f^mg(x)$, and let $\beta_m =
  \alpha_m \cdot f^mg(\alpha)$.  Then define a path $\gamma\co
  [0,\infty) \to X$ by:
  \[ \gamma = \beta_m \bigcup_{f^mg(\alpha_m)} f^mg(\beta_m)
  \bigcup_{(f^mg)^2(\alpha_m)} (f^mg)^2(\beta_m) \cdots \] By Lemma
  \ref{lm:quasi-geodesic}, $\gamma$ is an $L_m$--local
  $(1,\epsilon)$--quasi-geodesic and hence if $m \geq M$ then $\gamma$
  is a $(\lambda',\epsilon')$--quasi-geodesic from some $\lambda' \geq
  1$ and $\epsilon' \geq 0$ by Theorem \ref{th:local-qg}.  Therefore
  for any $x' \in X$ and $\ell \geq 0$ we have:
  \begin{align*}
    d(x',(f^mg)^\ell(x'))  & \geq d(x,(f^mg)^\ell(x)) - 2d(x',x) \\
    & \geq \frac{1}{\lambda'}L_m\ell - \epsilon' - 2d(x',x)
  \end{align*}
and hence $f^mg$ has nonzero translation length.
\end{proof}

The fully irreducible analog of Proposition \ref{prop:hyperbolic}
follows easily from Theorems \ref{th:complex} and
\ref{th:hyperbolic-action}.

\begin{proposition}\label{prop:fi}
  Let $\phi \in \Out F_k$ be a fully irreducible outer automorphism
  with stable and unstable trees $[T_+]$ and $[T_-]$ in
  $\overline{CV}_k$.  Suppose $\psi \in \Out F_k$ is such that
  $[T_+\psi] \neq [T_-]$.  Then there is an $M \geq 0$ such that $m
  \geq M$ the element $\phi^m\psi$ is fully irreducible.
\end{proposition}

\begin{proof}
  Let $\CX$ be the Bestvina--Feighn $\delta$--hyperbolic complex from
  Theorem \ref{th:complex} using $\phi_1 = \phi$ and let $A_+$ and
  $A_-$ denote the attracting and repelling fixed points of $\phi$ in
  $\partial \CX$.  What needs to be shown in order to apply Theorem
  \ref{th:hyperbolic-action} is that $[T_+\psi] \neq [T_-]$ implies
  that $A_+g \neq A_-$.  As the action of $\phi$ on $\CX$ satisfies
  WPD \cite[Proposition~4.27]{un:BF}, if $A_+\psi = A_-$ then for some
  $r,s > 0$ we have $\psi\phi^r\psi^{-1} = \phi^{-s}$
  \cite[Proposition~6]{ar:BF02}.  As the stable and unstable tree for
  positive powers of $\phi$ are the same as for $\phi$, this would
  imply $[T_+\psi] = [T_-]$.

  Now we can apply Theorem \ref{th:hyperbolic-action} to the pair
  $\phi$ and $\psi$ acting on $\CX$ to conclude that for large enough
  $m$, the element $\phi^m\psi$ does not have a bounded orbit and
  hence by \ref{th:complex} is fully irreducible.
\end{proof}

We would like to thank Mladen Bestvina for suggesting the use of WPD
in the above argument.


\section{The stable current for a product of twists}\label{sc:stable}

In this section we examine the qualitative behavior of the stable and
unstable currents associated to a product of Dehn twists.  The main
result is Corollary \ref{co:stable-limit} which produces elements of
$\Out F_k$ satisfying the hypotheses of Propositions
\ref{prop:hyperbolic} and \ref{prop:fi}.  We begin with a simple lemma
describing the change of a conjugacy class in $F_k$ resulting from
powers of a single twist.

\begin{lemma}\label{lm:inequalites}
  Let $T_1$ and $T_2$ be very small cyclic trees with edge stabilizers
  $c_1$ and $c_2$ and associated Dehn twists $\delta_1$ and
  $\delta_2$.  Suppose $\CT_1$ and $\CT_2$ are bases relative to $T_1$
  and $T_2$ respectively such that $c_2$ is cyclically reduced with
  respect to $\CT_1$ and $C$ is the constant from Lemma
  \ref{lm:occurance-cancellation} using these bases.  Then for any $x
  \in F_k$ and $n \geq r > 0$, the following hold.
  \begin{align}
    \la c_1^{\pm r}, \delta_1^n(x) \ra_{\CT_1} & \geq
    (n-r+1)\ell_{T_1}(x) -
    \la c_1^{\pm 1}, x \ra_{\CT_1} \label{al:1} \\
    \ell_{\CT_1}(\delta_1^n(x)) &
    \leq n\ell_{T_1}(x) + \ell_{\CT_1}(x)\label{al:3} \\
    \ell_{\CT_1}(\delta_1^n(x)) & \geq n\ell_{T_1}(x) +
    \ell_{\CT_1}(x) - \la
    c_1^{\pm 1}, x \ra_{\CT_1}\label{al:4} \\
    \la c_1^{\pm 1}, \delta_2^{-n}(x) \ra_{\CT_1} & \leq\ell_{T_2}(x)
    \bigl[ n\la c_1^{\pm 1},c_2 \ra_{\CT_1} + 2C \bigr] + \la c_1^{\pm
      1},x
    \ra_{\CT_1}\label{al:2} \\
    \ell_{\CT_1}(\delta_2^{-n}(x)) & \leq \ell_{T_2}(x)\bigl[
    n\ell_{\CT_1}(c_2) + 2C \bigr] + \ell_{\CT_1}(x)\label{al:5}
  \end{align}
\end{lemma}

\begin{proof}
  We begin by proving the first three inequalities. By replacing $x$
  by a conjugate we are free to assume that $x$ is $T_1$--reduced as
  all of the quantities involved in the inequalities only depend on the
  conjugacy class of $x$.  If $T_1$ is dual to an amalgamated free
  product we have:
  \[ x = x_1c_1^{i_1}y_1c_1^{j_1}\cdots x_mc_1^{i_m}y_mc_1^{j_m}. \]
  Therefore:
  \[ \delta_1^n(x) = x_1c_1^{i_1 + n}y_1c_1^{j_1 - n}\cdots
  x_mc_1^{i_m + n}y_mc_1^{j_m - n} \] is a cyclically reduced word in
  $\CT_1$.  Hence by only counting the occurrences of $c_1^{\pm r}$
  that appear in the $c_1^{i_s + n}$ and $c_1^{j_s - n}$ we see:
  \begin{align*}
    \la c_1^{\pm r}, \delta_1^n(x) \ra_{\CT_1} & \geq \sum_{s = 1}^m
    \bigl(|i_s + n| - r + 1 \bigr) + \bigl( |j_s - n| - r + 1 \bigr) \\
    & \geq 2m(n - r + 1) - \sum_{s = 1}^m |i_s| + |j_s| \\
    & \geq (n - r + 1)\ell_{T_1}(x) - \la c_1^{\pm 1},x
    \ra_{\CT_1}.
  \end{align*}
  A similar proof works if $T_1$ is dual to an HNN-extension.  This
  shows \eqref{al:1}; the inequalities \eqref{al:3} and \eqref{al:4}
  follow similarly by looking at the given cyclically reduced
  expression for $\delta_1^n(x)$.

  We now prove the last two inequalities.  As before, by replacing $x$
  by a conjugate we are free to assume that $x$ is $T_2$--reduced.  If
  $T_2$ is dual to an amalgamated free product we have:
  \[ x = x_1c_2^{i_1}y_1c_2^{j_1}\cdots x_mc_2^{i_m}y_mc_2^{j_m}. \]
  Therefore:
  \[ \delta_2^{-n}(x) = x_1c_2^{i_1 - n}y_1c_2^{j_1 + n}\cdots
  x_mc_2^{i_m - n}y_mc_2^{j_m + n} \] is a cyclically reduced word in
  $\CT_2$.  Hence by counting the number of occurrences of $c_1^{\pm
    1}$ in the various $x_s$, $y_s$, $c_2^{i_s-n}$ and $c_2^{j_s + n}$
  we see:
  \begin{align*}
    \la c_1^{\pm 1}, \delta_2^{-n}(x) \ra_{\CT_1} & \leq \ell_{T_2}(x)
    o( c_1^{\pm 1}, c_2^n )_{\CT_1}  \\
    & \qquad + \sum_{s = 1}^m o( c_1^{\pm 1}, x_s )_{\CT_1} + o(
    c_1^{\pm 1},c_2^{i_s} )_{\CT_1} + o( c_1^{\pm 1}, y_s
    )_{\CT_1} + o( c_1^{\pm 1} , c_2^{j_s} )_{\CT_1} \\
    & \leq \ell_{T_2}(x)n\la c_1^{\pm 1},c_2 \ra_{\CT_1} + \la
    c_1^{\pm
      1}, x \ra_{\CT_1} + 4mC \\
    & \leq \ell_{T_2}(x) \bigl[ n\la c_1^{\pm 1},c_2
    \ra_{\CT_1} + 2C \bigr] + \la c_1^{\pm 1},x \ra_{\CT_1}.
  \end{align*}
  A similar proof works if $T_1$ is dual to an HNN-extension.  This
  shows \eqref{al:2}.  The inequality \eqref{al:5} is just an
  application of the bounded cancellation lemma using the cyclically
  reduced expression for $\delta_2^{-n}(x)$.
\end{proof}

These estimates allow us to show our main technical result concerning
the stable currents.

\begin{theorem}\label{th:B}
  Let $T_1$ and $T_2$ be very small cyclic trees that fill, with edge
  stabilizers $c_1$ and $c_2$ and associated Dehn twist automorphisms
  $\delta_1$ and $\delta_2$. Let $N \geq 0$ be such that for $n
  \geq N$, we have that $\delta_1^n\delta_2^{-n}$ is a hyperbolic
  fully irreducible outer automorphism with stable and unstable
  currents $[\mu_+^n]$ and $[\mu_-^n]$ in $\BP Curr(F_k)$.  Then: \[
  \lim_{n \to \infty} [\mu_+^n] = [ \eta_{c_1}] \mbox{ and } \lim_{n
    \to \infty} [\mu_-^n] = [\eta_{c_2}]. \]
\end{theorem}

\begin{proof}
  Let $\CT_1$ be a basis for $F_k$ relative to $T_1$.  Denote by
  $\phi_n = \delta_1^n \delta_2^{-n}$.  Fix an element $a \in \CT_1$,
  denote its conjugacy class by $\alpha$, and denote $\phi_n^m(\alpha)$
  by $\alpha_n^m$.  Hence $\phi_n^m\eta_{\alpha} = \eta_{\alpha_n^m}$.
  As $\ell_{T_2}(\alpha) > \ell_{T_1}(\alpha) = 0$, Lemma 5.2 of
  \cite{un:CP} shows that since $n$ is sufficiently large, for $m \geq
  0$ we have $\ell_{T_2}(\alpha_n^m) \geq \ell_{T_1}(\alpha_n^m)$.
  Let $K > 0$ be such that for all $m,n \geq 0$:
  \[\ell_{\CT_1}(\alpha_n^m) \leq K(\ell_{T_1}(\alpha_n^m) +
  \ell_{T_2}(\alpha_n^m)) \leq 2K\ell_{T_2}(\alpha_n^m).\] Such a $K$
  exists by \cite[Theorem~1.4]{KL_Intersection}.

  Then for each $n \geq N$, as $\mu_+^n$ is the stable current for
  $\phi_n$, from the North-South dynamics of $\phi_n$ on $\BP
  Curr(F_k)$ (Theorem \ref{th:NS-curr}), we have for any $g \in F_k$
  and $\epsilon > 0$, a constant $M = M(n,g,\frac{\epsilon}{2})$ such
  that for $m \geq M$ (Lemma \ref{lm:convergence}):
  \begin{equation}\label{eq:rational-approx}
    \left|
      \frac{\la g, \eta_{\alpha_n^m} \ra_{\CT_1} }
      {\omega_{\CT_1}(\eta_{\alpha_n^m})} -  
      \frac{\la g, \mu_+^n \ra_{\CT_1} }
      {\omega_{\CT_1}(\mu_+^n)}
    \right| < \frac{\epsilon}{2}.
  \end{equation}
  
  We analyze how the current $\eta_{\alpha_n^m}$ changes as $m \to
  \infty$ in terms of $n$.  Fix a basis $\CT_2$ that is relative to
  $T_2$ such that $c_2$ is cyclically reduced with respect to $\CT_1$,
  and let $C$ be larger than either constant $C = C(T_1,T_2)$ from
  Theorem \ref{th:linear} or the constant $C$ from Lemma
  \ref{lm:occurance-cancellation}, using the bases $\CT_1$ and $\CT_2$.
  Applying \eqref{al:twist-lb}, \eqref{al:1} and \eqref{al:2}, for any
  $m \geq 0$ and $n \geq r > 0$ we have:
  \begin{align*}
    \la c_1^{\pm r}, \delta_1^n\delta_2^{-n}(\alpha_n^m) \ra_{\CT_1} &
    \geq (n-r+1)\ell_{T_1}(\delta_2^{-n}(\alpha_n^m)) - \la
    c_1^{\pm 1}, \delta_2^{-n}(\alpha_n^m)  \ra_{\CT_1} \\
    & \geq (n-r+1)\Bigl[\ell_{T_2}(\alpha_n^m)\bigl[ n\ell_{T_1}(c_2)
    - (C + 1) \bigr] \Bigr] \\
    & \qquad - \Bigl[ \ell_{T_2}(\alpha_n^m)\bigl[n\la c_1^{\pm 1},
    c_2 \ra_{\CT_1} + 2C\bigr] +
    \la c_1^{\pm 1}, \alpha_n^m \ra_{\CT_1} \Bigr] \\
    & \geq n^2\ell_{T_2}(\alpha_n^m)\ell_{T_1}(c_2) \\
    & \qquad - n\ell_{T_2}(\alpha_n^m)\bigl[ (C+1) +
    (r-1)\ell_{T_1}(c_2) + \la c_1^{\pm 1},c_2 \ra_{\CT_1} \bigr] \\
    & \qquad - 2C\ell_{T_1}(\alpha^m_n) - \ell_{\CT_1}(\alpha_n^m) \\
    & \geq \ell_{T_2}(\alpha_n^m)\Bigl[n^2\ell_{T_1}(c_2)  \\
    & \qquad - n\bigl[ (C+1) + (r-1)\ell_{T_1}(c_2) + \la c_1^{\pm
      1},c_2 \ra_{\CT_1} \bigr] -
    (2C + 2K) \Bigr] \\
    & \geq \ell_{T_2}(\alpha_n^m)\bigl[n^2\ell_{T_1}(c_2) - nC_1 - C_2
    \bigr]
  \end{align*}
  for some constants $C_1 \geq 0$ and $C_2 \geq 0$ that do not depend
  on $m$.  Applying \eqref{al:twist-ub}, \eqref{al:3} and \eqref{al:5}
  we also have for any $m \geq 0$ and $n > 0$:
  \begin{align*}
    \ell_{\CT_1}(\delta_1^n\delta_2^{-n}(\alpha_n^m)) & \leq
    n\ell_{T_1}(\delta_2^{-n}(\alpha_n^m)) +
    \ell_{\CT_1}(\delta_2^{-n}(\alpha_n^m)) \\
    & \leq n\Bigl[\ell_{T_2}(\alpha_n^m)\bigl[n\ell_{T_1}(c_2) +
    (C + 1) \bigr] \Bigr] \\
    & \qquad + \ell_{T_2}(\alpha_n^m)\bigl[
    n\ell_{\CT_1}(c_2) + 2C \bigr] + \ell_{\CT_1}(\alpha_n^m) \\
    & \leq n^2\ell_{T_2}(x)\ell_{T_1}(c_2) \\
    & \qquad + n\ell_{T_2}(\alpha_n^m)\bigl[ (C+1) +
    \ell_{\CT_1}(c_2) \bigr] + (2C + 2K)\ell_{T_1}(\alpha_n^m) \\
    & \leq \ell_{T_2}(\alpha_n^m) \bigl[n^2\ell_{T_1}(c_2) + nC'_1
    + C'_2 \bigr]
  \end{align*}
  for some constants $C'_1 \geq 0$ and $C'_2 \geq 0$ that do not
  depend on $m$.  Therefore, given $r > 0$ there are constants
  $\beta_1 \geq 0$ and $\beta_2 \geq 0$ that do not depend on $m$ such
  that for any $n \geq r$:
  \begin{equation}\label{eq:numerator}
    \ell_{\CT_1}(\delta_1^n\delta_2^{-n}(\alpha_n^m)) - 
    \la c_1^{\pm r}, \delta_1^n\delta_2^{-n}(\alpha_n^m) \ra_{\CT_1} \leq 
    \ell_{T_2}(\alpha_n^m)\bigl[n\beta_1 + \beta_2\bigr].
  \end{equation}
  Now, applying \eqref{al:twist-lb} and \eqref{al:4}, we have for any
  $m \geq 0$ and $n > 0$:
  \begin{align*}
    \ell_{\CT_1}(\delta_1^n\delta_2^{-n}(\alpha_n^m)) & \geq
    n\ell_{T_1}(\delta_2^{-n}(\alpha_n^m)) +
    \ell_{\CT_1}(\delta_2^{-n}(\alpha_n^m)) -
    \la c_1^{\pm 1}, \delta_2^{-n}(\alpha_n^m) \ra_{\CT_1} \\
    & \geq n\Bigl[\ell_{T_2}(\alpha_n^m)\bigl[n\ell_{T_1}(c_2)
    - (C+1) \bigr] \Bigr] \\
    & \qquad - \ell_{T_2}(\alpha_n^m)\bigl[n\ell_{\CT_1}(c_2) + 2C
    \bigr] - \la c_1^{\pm 1}, \alpha_n^m
    \ra_{\CT_1} \\
    & \geq n^2\ell_{T_2}(\alpha_n^m)\ell_{T_1}(c_2) \\
    & \qquad - n\ell_{T_2}(\alpha_n^m)\bigl[(C+1) + \ell_{\CT_1}(c_2)
    \bigr] - (2C+2K)\ell_{T_2}(\alpha_n^m)
  \end{align*}
  Therefore, there are constants $\gamma_1 \geq 0$, $\gamma_2 \geq 0$
  and $\gamma_3 \geq 0$ that do not depend on $m$ such that for $n >
  0$:
  \begin{equation}\label{eq:denominator}
    \ell_{\CT_1}(\delta_1^n\delta_2^{-n}(\alpha_n^m)) \geq 
    \ell_{T_2}(\alpha_n^m) \bigl[n^2\gamma_1 - n\gamma_2 - 
    \gamma_3\bigr].
  \end{equation}
  
  As a first approximation, we will show that the currents
  $\eta_{\alpha^m_n}$ converge to the correct value on
  $Cyl_{\CT_1}(c_1^r)$.  Notice $\eta_{c_1}(Cyl_{\CT_1}(c_1^r)) = 1$.
  Suppose $g = c_1^{\pm r}$ for some $r > 0$.  Let $\epsilon > 0$ and
  fix $n \geq \max\{ N,r\}$ large enough such that
  $\epsilon(n^2\gamma_1 - n\gamma_2 - \gamma_3) > 2(n\beta_1 +
  \beta_2)$.  Now let $m \geq M(n,g,\frac{\epsilon}{2})$.  Then:
  \begin{align}
    \left| \frac{\la g,\eta_{c_1}
        \ra_{\CT_1}}{\omega_{\CT_1}(\eta_{c_1})} - \frac{\la g,
        \mu_+^n\ra_{\CT_1}}{\omega_{\CT_1}(\mu_+^n)} \right| & \leq
    \left| \frac{\la g,\eta_{c_1}
        \ra_{\CT_1}}{\omega_{\CT_1}(\eta_{c_1})} - \frac{\la g,
        \eta_{\alpha_n^{m+1}}\ra_{\CT_1}}{\omega_{\CT_1}(\eta_{\alpha_n^{m+1}})}
    \right|+ \left| \frac{\la g,\eta_{\alpha_n^{m+1}}
        \ra_{\CT_1}}{\omega_{\CT_1}(\eta_{\alpha_n^{m+1}})} -
      \frac{\la g,
        \mu_+^n\ra_{\CT_1}}{\omega_{\CT_1}(\mu_+^n)} \right| \notag \\
    & < \left| 1 - \frac{\la g,
        \eta_{\alpha_n^{m+1}}\ra_{\CT_1}}{\omega_{\CT_1}(\eta_{\alpha_n^{m+1}})}
    \right| + \frac{\epsilon}{2} \notag \\
    & = \left| \frac{ \ell_{\CT_1}(\delta_1^n\delta_2^{-n}(\alpha_n^m)) -
        \la c_1^{\pm r}, \delta_1^n\delta_2^{-n}(\alpha_n^m)
        \ra_{\CT_1}}{\ell_{\CT_1}(\delta_1^n\delta_2^{-n}(\alpha_n^m))}
    \right| + \frac{\epsilon}{2} \notag \\
    & \leq \left| \frac{\ell_{T_2}(\alpha_n^m)\bigl[ n\beta_1
        + \beta_2 \bigr]}{\ell_{T_2}(\alpha_n^m)\bigl[
        n^2\gamma_1 - n\gamma_2 - \gamma_3 \bigr]} \right| +
    \frac{\epsilon}{2} \notag \\
    & < \frac{\epsilon}{2} + \frac{\epsilon}{2} =
    \epsilon.\label{al:measure-c}
  \end{align}
  Now suppose $g \neq c_1^{\pm r}$ for any $r > 0$; in this case
  $\eta_{c_1}(Cyl_{\CT_1}(g)) = 0$.  There is some $a_0 \in \CT_1 - \{
  c_1 \}$ such that $\la a_0^{\pm 1}, g \ra_{\CT_1} > 0$.  Therefore
  for any $m \geq 0$:
  \begin{equation}\label{eq:measure-g}
    \la a_0^{\pm 1}, \alpha_n^m \ra_{\CT_1}\la a_0^{\pm 1}, g \ra_{\CT_1} 
    \geq \la g^{\pm 1}, \alpha_n^m \ra_{\CT_1}
  \end{equation}
  as every occurrence of $g^{\pm 1}$ in $\alpha_n^m$ contains some
  occurrence of $a_0^{\pm 1}$ in $\alpha_n^m$ and such an occurrence
  can only be used $\la a_0^{\pm 1}, g\ra_{\CT_1}$ times. Since:
  \begin{equation*}
    \frac{1}{\ell_{\CT_1}(\alpha_n^m)}\sum_{x \in \CT_1 - \{ c_1 \}} 
    \la x^{\pm 1},\alpha_n^m \ra_{\CT_1} = 1 - \frac{\la c_1^{\pm 1}, \alpha_n^m 
      \ra_{\CT_1}}{\ell_{\CT_1}(\alpha_n^m)}
  \end{equation*}
  the computation in \eqref{al:measure-c} combined with
  \eqref{eq:measure-g} shows that there is an $n = n(g,\epsilon)$ such
  that $2\la g^{\pm 1}, \alpha_n^m \ra_{\CT_1} <
  \epsilon\ell_{\CT_1}(\alpha_n^m)$ for $m$ sufficiently large.  Then
  for $m \geq M(n,g,\frac{\epsilon}{2})$:
  \begin{align}
    \left| \frac{\la g,\eta_{c_1}
        \ra_{\CT_1}}{\omega_{\CT_1}(\eta_{c_1})} - \frac{\la g,
        \mu_+^n\ra_{\CT_1}}{\omega_{\CT_1}(\mu_+^n)} \right| & \leq
    \left| \frac{\la g,\eta_{c_1}
        \ra_{\CT_1}}{\omega_{\CT_1}(\eta_{c_1})} - \frac{\la g,
        \eta_{\alpha_n^m}\ra_{\CT_1}}{\omega_{\CT_1}(\eta_{\alpha_n^m})}
    \right|+ \left| \frac{\la g,\eta_{\alpha_n^m}
        \ra_{\CT_1}}{\omega_{\CT_1}(\eta_{\alpha_n^m})} - \frac{\la g,
        \mu_+^n\ra_{\CT_1}}{\omega_{\CT_1}(\mu_+^n)} \right| \notag \\
    & < \left| \frac{\la g, \eta_{\alpha_n^m}
        \ra_{\CT_1}}{\omega_{\CT_1}(\eta_{\alpha_n^m})}
    \right| + \frac{\epsilon}{2} \notag \\
    & = \left| \frac{\la g^{\pm 1},
        \alpha_n^m\ra_{\CT_1}}{\ell_{\CT_1}(\alpha_n^m)} \right| +
    \frac{\epsilon}{2} \notag \\
    & < \frac{\epsilon}{2} + \frac{\epsilon}{2} =
    \epsilon.\label{al:measure-g}
  \end{align}

  Putting together \eqref{al:measure-c} and \eqref{al:measure-g}, we
  have that $\lim_{n \to \infty} [\mu_+^n] = [\eta_{c_1}]$.  The same
  argument applied to $\phi^{-1}_n$ shows that $\lim_{n \to \infty}
  [\mu_-^n] = [\eta_{c_2}]$.
\end{proof}

\begin{remark}\label{rm:analogy}
  We remark that Theorem \ref{th:B} is analogous to the surface
  setting.  Given two simple closed curves $\alpha,\beta \subset S_g$
  that fill, the stable and unstable measured laminations
  $[\Lambda_+^n]$ and $[\Lambda_-^n]$ in the Thurston boundary of
  Teichm\"uller space associated to the pseudo-Anosov mapping classes
  $\delta_\alpha^n\delta_\beta^{-n}$ converge to $[\alpha]$ and
  $[\beta]$ respectively.  Here $\delta_\alpha$ and $\delta_\beta$ are
  the respective Dehn twist homeomorphisms about $\alpha$ and $\beta$.

  This raises a subtle point.  To the hyperbolic fully irreducible
  outer automorphisms $\delta_1^n\delta_2^{-n}$ in Theorem \ref{th:B}
  are also associated the stable and unstable trees $[T_+^n]$ and
  $[T_-^n]$ in $\overline{CV}_k$ (see section \ref{ssc:outer}).  As
  $\overline{CV}_k$ is compact, the associated sequences $\{[T_+^n]\}$
  and $\{[T_-^n]\}$ have accumulation points.  But in contrast with
  Theorem \ref{th:B}, it is not clear whether there is a single
  accumulation point for each respective sequence or how to
  characterize an accumulation point for either sequence.  By Theorems
  \ref{th:intersection}, \ref{th:ergodic} and \ref{th:B}, the element
  $c_2$ has a fixed point in any accumulation point of $\{[T_+^n]\}$,
  and similarly $c_1$ has a fixed point in any accumulation point of
  $\{[T_-^n]\}$.  However it is unlikely that this is a
  characterization of the accumulation points for the sequences
  $\{[T_+^n]\}$ and $\{[T_-^n]\}$.
\end{remark}

The following Corollary is essential for our main theorem (Theorem
\ref{th:A}).

\begin{corollary}\label{co:stable-limit}
  Let $T_1$ and $T_2$ be very small cyclic trees that fill, with edge
  stabilizers $c_1$ and $c_2$ and associated Dehn twist automorphisms
  $\delta_1$ and $\delta_2$. Let $N \geq 0$ be such that for $n
  \geq N$, we have that $\delta_1^n\delta_2^{-n}$ is a hyperbolic
  fully irreducible outer automorphism with stable and unstable
  currents $[\mu_+^n]$ and $[\mu_-^n]$ in $\BP Curr(F_k)$ and stable
  and unstable trees $[T_+^n]$ and $[T_-^n]$ in $\overline{CV}_k$.
  For $\psi \in \Out F_k$ such that the conjugacy class of $\psi(c_1)$ is
  not equal to the conjugacy class of $c_2$, there is an $N_1 \geq N$
  such that for $n \geq N_1$ we have $[\psi\mu_+^n] \neq [\mu_-^n]$
  and $[T_+^n\psi] \neq [T_-^n]$.
\end{corollary}

\begin{proof}
  As the conjugacy class of $\psi(c_1)$ is not equal to the conjugacy
  class of $c_2$ we have that $[\psi\eta_{c_1}] \neq [\eta_{c_2}]$.
  Fix disjoint open sets $U_1$ and $U_2$ of $\BP Curr(F_k)$ containing
  $[\psi\eta_{c_1}]$ and $[\eta_{c_2}]$ respectively.  By Theorem
  \ref{th:B}, there is an $N_1$ such that for $n \geq N_1$ we have:
  $[\psi\mu_+^n] \in U_1$ and $[\mu_-^n] \in U_2$, and hence as $U_1$
  and $U_2$ are disjoint, $[\psi\mu_+^n] \neq [\mu_-^n]$.

  Additionally for $n \geq N_1$ we have $\la T_+^n\psi, \mu_+^n \ra =
  \la T_+^n, \psi\mu_+^n \ra > 0$ by Theorem \ref{th:ergodic}, as
  $[\psi\mu_+^n] \neq [\mu_-^n]$.  As $\la T_-^n,\mu_+^n \ra = 0$, 
  this shows that $[T_+^n\psi] \neq [T_-^n]$.
\end{proof}


\section{\texorpdfstring{A hyperbolic fully irreducible automorphism
    for every matrix in $\GL(k,\BZ)$}{A hyperbolic fully irreducible
    automorphism for every matrix in GL(k,Z)}}\label{sc:surject}
  
Our main theorem now follows easily.

\begin{theorem}\label{th:A}
  Suppose $k \geq 3$.  For any $A \in \GL(k,\BZ)$, there is a
  hyperbolic fully irreducible outer automorphism $\phi \in \Out F_k$
  such that $\phi_* = A$.
\end{theorem}

\begin{proof}
  Fix $\psi \in \Out F_k$ such that $\psi_* = A$.  Let $T$ be a very
  small cyclic tree dual to an amalgamated free product with edge
  stabilizer $c_1$ (a primitive element of $F_k$) and associated Dehn
  twist $\delta_1$.  As is shown in \cite[Remark 2.7]{un:CP} given any
  hyperbolic fully irreducible automorphism $\theta \in \Out F_k$, the
  pair $T$ and $T\theta^\ell$ fill for sufficiently large $\ell$.  The
  edge stabilizer for $T\theta^\ell$ is $\theta^{-\ell}(c_1)$.  Thus
  for large enough $\ell$ we can assure that the very small cyclic
  trees $T$ and $T\theta^\ell$ fill and that the conjugacy class of
  $\psi(c_1)$ is not equal to the conjugacy class of
  $\theta^{-\ell}(c_1)$ (the edge stabilizer for $T\theta^\ell$).

  Let $\delta_2$ be the associated Dehn twist for $T\theta^\ell$.
  By Theorem \ref{th:twist}, Propositions \ref{prop:hyperbolic}
  and \ref{prop:fi} and Corollary \ref{co:stable-limit}, for large $m$
  and $n$ the outer automorphism $(\delta_1^n\delta_2^{-n})^m\psi$ is
  a hyperbolic fully irreducible element.  Since both $\delta_1$ and
  $\delta_2$ act trivially on $H_1(F_k,\BZ)$, we have
  $((\delta_1^n\delta_2^{-n})^m\psi)_* = \psi_* = A$.
\end{proof}

\bibliography{bibliography}

\bibliographystyle{siam}

\end{document}